\documentclass[10pt]{amsart}

\title{Projections of sets with optimal oracles onto $k$-planes}

\author{Jacob B. Fiedler}  
\address{Department of Mathematics, University of Wisconsin-Madison, Wisconsin 53715}

\email{jbfiedler2@wisc.edu}
\thanks{The first author was supported in part by NSF DMS-2037851 and NSF DMS-2246906.}

\author{Zhifan Jing}

\address{Department of Computer Science, University of Wisconsin-Madison, Wisconsin 53715}

\email{zjing24@wisc.edu}


\usepackage{enumitem}
\usepackage{amsmath}
\usepackage{amssymb}
\usepackage{mdframed}
\usepackage{amsthm}
\usepackage{mathtools}
\usepackage{enumitem}
\usepackage{placeins}
\usepackage{url}
\usepackage{hyperref}

\newtheorem{thm}{Theorem}
\newtheorem{obs}[thm]{Observation}
\newtheorem{lem}[thm]{Lemma}
\newtheorem{prop}[thm]{Proposition}

\newtheorem{defn}[thm]{Definition}

\newtheorem*{T1}{Theorem~\ref{thm:MarstrandsProjection}}
\newtheorem*{T2}{Theorem~\ref{thm:exceptionalSetEstimate}}

\theoremstyle{remark}

\newcommand{\R}{\mathbb{R}}

\newcommand{\N}{\mathbb{N}}
\newcommand{\Q}{\mathbb{Q}}
\newcommand{\G}{\mathcal{G}}

\newcommand{\ve}{\varepsilon}

\begin{document}
\maketitle

\begin{abstract}
    We prove a Kaufman-type exceptional set estimate for sets in $\mathbb{R}^n$ that have optimal oracles, a class of sets that strictly contains the analytic sets and sets with equal Hausdorff and packing dimension. As a consequence, we generalize the conditions under which Marstrand's projection theorem for $k$-planes is known to hold. Our proofs use effective methods, especially Kolmogorov complexity, and along the way, we introduce several new tools for studying the information content of elements of the Grassmannian. 
\end{abstract}

\section{Introduction}

Let $p_VF$ denote the projection of a set $F$ onto a subspace $V$ in the Grassmannian $\mathcal{G}(n, k)$. Marstrand famously proved that for any analytic set \(F \subseteq \mathbb{R}^2\),
\[
\dim_H(p_eF)=\min\{\dim_H(F),1\}
\]
for almost every \(e\in S^1\) \cite{Marstrand1954Projection}. Mattila later generalized this theorem to \(\mathbb{R}^n\), showing that for any analytic set \(F \subseteq \mathbb{R}^n\),
\[
\dim_H(p_VF)=\min\{\dim_H(F),k\}
\]
for almost every \(V\in \G(n,k)\) \cite{mattila1975exceptionalset}. These results formalize the intuition that the projection of a (suitably well-behaved) set should be large in most directions, and they helped initiate an extensive study of fractal projections that has come to span numerous fields of math \cite{FalFraJin15, falconer2026seventyyearsfractalprojections}.

Our first theorem extends the above result to a broader class of sets, namely sets with optimal oracles.
\begin{thm}\label{thm:MarstrandsProjection}
    Let $F \subseteq \R^n$ be a set with optimal oracles. Then for almost every $V \in \G(n,k)$, \begin{equation*}
        \dim_H(p_VF)= \min\{\dim_H (F), k\}
    \end{equation*}
\end{thm}

We defer a formal definition of sets with optimal oracles to Section~2, but note that this class was introduced by Stull in the context of projection theorems \cite{Stull22a} and contains (among other sets) the analytic sets, sets with equal Hausdorff and packing dimension, sets with positive finite measure according to certain metric outer measures, as well as \emph{every} set assuming the axiom of determinacy (c.f. \cite{CroFishJack22}). To give a brief indication of the history of this notion, N. Lutz and Stull proved the $k=1$ case of Marstrand's projection theorems for sets with equal Hausdorff and packing dimension using algorithmic methods \cite{LutStu18Projections}. Orponen then gave a combinatorial proof for general $k$ \cite{Orponen2021Projections}. Finally, Stull introduced sets with optimal oracles and proved Marstrand's projection theorem in $\mathbb{R}^2$ for them.

Theorem \ref{thm:MarstrandsProjection} is a consequence of a stronger ``exceptional set'' estimate. Marstrand's projection theorem says that the directions in which the projection of a set fails to have maximal size are a set of measure zero. It is reasonable to ask whether the set of directions such that a set's projections have \emph{much} less than maximal size is even smaller. Formally, for $F \subseteq \R^n$, define the exceptional set \begin{equation*}
     E_s(F):=\{V\in\mathcal{G}(n, k): \dim_H(p_V F)<s\}.
\end{equation*}

It was first proved in $\R^2$ by Kaufman \cite{Kaufman1968ExceptionalSetEstimate}  and then extended to $\R^n$ by Mattila \cite{mattila1975exceptionalset} that for any analytic $F\subset \R^n$, \begin{equation*}
    \dim_H(E_s(F)) \leq k(n-k) + s - k.
\end{equation*}
In this paper, we extend this bound to sets with optimal oracles.
\begin{thm}\label{thm:exceptionalSetEstimate}
      Let $F\subseteq\mathbb{R}^n$ with optimal oracles and $k<n$ be given. If $F$ has Hausdorff dimension $a \geq s$, then
    \begin{equation*}
        \dim_H(E_s(F))\leq k(n-k) + s - k.
    \end{equation*}
\end{thm}
Note that, in the planar case, the same estimate for sets with optimal oracles was established by the first author and Stull \cite{FieStu24Projections}.

For analytic sets, this exceptional set estimate is not state of the art. As a consequence of their sharp bounds on Furstenberg sets in the plane, Ren and Wang established sharp exceptional set estimates for projections of analytic sets in $\mathbb{R}^2$ \cite{ren2023furstenberg} (see also \cite{CsornStull2025A}). In higher dimensions, Falconer established a bound complementing Mattila's Kaufman-type result \cite{Falconer82Projections}. Gan proved additional bounds which are sharp for some parameters \cite{gan2024exceptionalsets}. Recently, Cholak-Cs\"ornyei-Lutz-Lutz-Mayordomo-Stull improved Gan's bounds, in particular resolving the problem for projections of analytic sets onto lines and hyperplanes \cite{CholakCsorn2025Bourgain}.

We briefly comment on the proof of the main theorem. We employ ``effective'' methods, which establish classical results in geometric measure theory by examining the information content of points in sets. Our classical projection theorem is a consequence of an effective theorem which states that, under the right conditions, projections of high complexity points onto high complexity $k$-planes have high complexity themselves. Overall, we argue similarly to \cite{LutStu18Projections}, but we require additional geometric tools to handle $k$-planes, as well as generalizations of existing algorithmic results. The first author introduced several of these tools along with a notion of Kolmogorov complexity on the Grassmannian in the context of higher-dimensional Furstenberg-type sets \cite{fiedler2025extensionsUnions}. We will provide a brief overview of algorithmic methods in $\mathbb{R}^n$ and the Grassmannian, prove several additional results in the Grassmannian (that may be of independent interest), establish two important lemmas, and then apply these lemmas in the proof of our main theorem. 

\section{Preliminaries}

\subsection{Effective dimension}
 We provide a brief overview of Kolmogorov complexity. For further details on the Kolmogorov complexity of strings, we refer readers to \cite{downey2010, li2008introduction}, and for further details on complexity at precision $r$, we recommend \cite{lutz2018algorithmic, lutz2020bounding}.
 
 Let $\sigma$ and $\tau$ be two finite strings. Let $B$ be any oracle. We define the Kolmogorov complexity as follows
\begin{equation*}
    K^B(\sigma \mid \tau) = \min_{\pi \in \{0,1\}^*} \{\ell(\pi) : U^B(\pi,\tau) = \sigma\}.
\end{equation*}
where $U$ is some fixed (prefix-free) universal oracle Turing Machine. The complexity is just the length of the shortest input to \(U\) that produces \(\sigma\). Since Turing Machine can only process finite data objects, for most reals with infinite digits, we can only work with finite approximations. Fixing some standard encoding of the rational vectors in $\mathbb{R}^n$ as binary strings gives a notion of Kolmogorov complexity for rational vectors. This leads to the following definition of Kolmogorov complexity at precision $r\in \N$ of $x \in \R^n$: \begin{equation*}
    K_r^B(x) = \min \{K^B(p): p \in B_{2^{-r}}(x) \cap \Q^n\}.
\end{equation*}
The conditional Kolmogorov complexity of $x\in \R^n$ at precision $r$ given $y \in \R^m$ at precision $s\in \N$ is \begin{equation*}
    K^B_{r,s}(x\mid y) = \max\{\min\{K^B_r(p\mid q): p \in B_{2^{-r}}(x)\cap \Q^n\}: q\in B_{2^{-r}}(y)\cap\Q^n\}.
\end{equation*}

With oracle access, a Turing machine can query information at any desired precision, whereas under conditional access the precision of the information provided is fixed. So clearly oracle access is always no weaker than conditional access. Furthermore, granting access to additional oracles cannot increase complexity by more than a trivial amount, since the machine can always choose not to use them. The following lemma formalizes these two simple observations.
\begin{lem}\label{lem:conditionalAndOracle}
    Let $x,y \in \R^n$, $r,s\in \N$, and $A,B \subseteq \N$. We always have \begin{equation*}
        K_{r,s}(x \mid y) \geq K_r^y(x) - O(\log (r+s))
    \end{equation*}
    and \begin{equation*}
        K_r^{A}(x) \geq  K_r^{A,B}(x) - O(\log r).
    \end{equation*}
\end{lem}

Lutz and Stull proved that the symmetry of information holds for Kolmogorov complexity in Euclidean spaces \cite{lutz2020bounding}. The intuition is that computing \(x\) and \(y\) together is equivalent to first computing \(x\) and then computing \(y\) conditional on \(x\).
\begin{lem}
    For any $x\in \R^n, \ y\in \R^m$ and  $r,s\in \N$ with $r \geq s$, \begin{equation*}
    |K^B_{r,s}(x\mid y) + K^B_s(y) - K^B_{r,s}(x,y)| \leq O(\log r) + O(\log\log |y|).
    \end{equation*}
\end{lem}
The following lemma, established by Case and Lutz \cite{case2015dimension}, states that dimension of the ambient space times \(s\) bits of additional information is sufficient to compute the precision of a point to an additional \(s\) digits.
\begin{lem}\label{lem:CaseLutz}
    For any $x \in \R^n$ and any $r,s \in \N$, \begin{equation*}
        K_{r + s}(x) \leq K_r(x) + ns + O(\log r + \log s).
    \end{equation*}
\end{lem}

Next, we introduce two important technical lemmas which have been utilized in a number of effective arguments \cite{lutz2020bounding, stull2022pinned}; we state them in the form that they appear in \cite{LutStu18Projections}. 

\begin{lem}\label{lem:oracleD}
Let $z \in \R^n$, $\eta \in \Q \cap [0,\dim(z)]$, and $r \in \N$. Then there is an oracle $D = D(r,z,\eta)$ with the following properties.
\begin{enumerate}[label=(\roman*)]
    \item For every $t \le r$,
    \[
    K_t^D(z) = \min\{\eta r,\, K_t(z)\} + O(\log r).
    \]

    \item For every $m,t \in \mathbb{N}$ and $y \in \mathbb{R}^m$,
    \[
    K_{t,r}^D(y \mid z) = K_{t,r}(y \mid z) + O(\log r),
    \]
    and
    \[
    K_t^{z,D}(y) = K_t^z(y) + O(\log r).
    \]
\end{enumerate}
\end{lem}

\begin{lem}\label{lem:oracleD2}
Let $z \in \R^n$, $B \subseteq \N$, $\eta \in \Q \cap [0,\dim(z)]$, $\ve > 0$, and $r \in \N$. Let $D = D(r,z,\eta)$ be the oracle defined in Lemma \ref{lem:oracleD}. If
\begin{equation*}
K_r^B(z) \ge K_r(z) - \ve r,
\end{equation*}
then
\begin{equation*}
K_r^{B,D}(z) \ge K_r^D(z) - \ve r - O(\log r).
\end{equation*}
\end{lem}

Essentially, sometimes it will be useful for us to lower the complexity of a point at some precision. Morally speaking, our effective argument will show that the complexity of a projection is large by using that projection, along with oracle access to the $k$-plane $V$, to determine the point being projected, meaning the point cannot be much more complicated than its projection. However, this only works if $z$ is a ``simple'' point in the fiber $\{w: p_Vw=p_Vz\}$ (in a certain precise sense). Working relative to $D$, $z$ will satisfy the necessary conditions. The remaining content of Lemma \ref{lem:oracleD} and Lemma \ref{lem:oracleD2} is that working relative to this oracle does not unduly alter the complexity of \emph{other} objects.

J. Lutz introduced a notion of effective Hausdorff dimension for infinite binary sequences \cite{Lutz03b}, and Mayordomo subsequently gave a characterization of this notion in terms of Kolmogorov complexity, which we take to be the definition \cite{Mayordomo02}. Specifically, define \begin{equation*}
    \dim^A(x) := \liminf_{r \to \infty}\frac{K_r^A(x)}{r}.
\end{equation*}

The point-to-set principle of J. Lutz and N. Lutz is a crucial tool that connects effective dimension to classical dimension.

\begin{thm}[\cite{lutz2018algorithmic}]
    For every set $E \subseteq\R^n$, \begin{equation*}
        \dim_H(E) = \min_{B \subset \N}\sup_{x \in E}\dim^A(x).
    \end{equation*}
\end{thm}

We call any oracle testifying this equality a Hausdorff oracle for $E$. Every set has a Hausdorff oracle, but it will often be useful if a set possesses a special kind of Hausdorff oracle. The notion of a (Hausdorff) optimal oracle, which plays a central role in this paper, was introduced by Stull in \cite{Stull22a}.
\begin{defn}
Let \(E \subseteq \mathbb{R}^n\) and \(A \subseteq \mathbb{N}\). We say that \(A\) is \emph{Hausdorff optimal} for \(E\) if the following conditions are satisfied.
\begin{enumerate}
    \item \(A\) is a Hausdorff oracle for \(E\).
    \item For every \(B \subseteq \mathbb{N}\) and every \(\varepsilon > 0\) there is a point \(x \in E\) such that
    $\dim^{A,B}(x) \geq \dim_H(E) - \ve$
    and for every sufficiently large \(r \in \mathbb{N}\),
    \[
    K_r^{A,B}(x) \geq K_r^A(x) - \ve r.
    \]
\end{enumerate}
\end{defn}

In other words, in a set with optimal oracles, given any oracle $B$, we can always find points testifying to the dimension of the set with complexity that is essentially unaffected by $B$ at every precision. As mentioned, the class of sets that have an optimal oracle strictly contains the analytic sets, among others. 

\subsection{Effective dimension on the Grassmannian}

In \cite{fiedler2025extensionsUnions}, the first author introduced a notion of Kolmogorov complexity on the Grassmannian. We associate a $k$-plane through the origin to its unique orthogonal projection matrix. Then, we use the following metric on the Grassmannian,
\begin{equation*}
    \rho (V_1, V_2) = \sup_{x\in\mathbb{R}^{n}, \vert x\vert =1}\vert p_{V_1}x - p_{V_2}x\vert. 
\end{equation*}
The ``rational'' elements of the Grassmannian are $k$-planes with orthogonal projection matrices having all rational entries. A Turing machine can enumerate the rational orthogonal projection matrices, so it is reasonable to define a notion of Kolmogorov complexity for them. With the above metric and the fact that the rational projection matrices are dense, this is enough to define the complexity at precision $r$ of arbitrary $V\in\mathcal{G}(n, k)$. In particular, 
\begin{equation*}
    K_r^B(V) = \min\{K^B(Q): Q\in \mathcal{G}(n, k)\cap \mathbb{Q}^{n\times n}\cap B_{2^{-r}}(V)\}.
\end{equation*}
We define the effective Hausdorff and packing dimension of points in the obvious way. Furthermore, a general point-to-set principle of J. Lutz, N. Lutz, and Mayordomo holds in this setting \cite{LuLuMay2023PtS}. A similar sequence of definitions leads to notions of complexity in the affine Grassmannian $\mathcal{A}(n, k)\supset\mathcal{G}(n, k)$

The symmetry of information is a crucial tool in almost any application of effective methods, and \cite{fiedler2025extensionsUnions} establishes that it holds in the Grassmannian as well as between the Grassmannian and Euclidean space. Let $\mathcal{X}$ and $\mathcal{Y}$ be any Euclidean space, Grassmannian, or affine Grassmannian. Then we have the following.  
\begin{prop}\label{prop:symmetry}
    Let $X\in\mathcal{X}$ and $Y\in\mathcal{Y}$. For every $B\subseteq\mathbb{N}$ and precisions $r, s$, 
    \begin{equation*}
        K^B_{r, s}(X, Y) = K^B_s(Y) + K_{r, s}^B(X\mid Y) \pm O(\log (r + s)).
    \end{equation*}
\end{prop}

The point-to-set principle and the symmetry of information are key algorithmic tools. However, to prove certain theorems in the Grassmannian, it is also essential to have a number of more geometric tools. First, knowing points that span a $k$-plane is enough to determine the $k$-plane.
\begin{lem}[\cite{fiedler2025extensionsUnions}]\label{lem:pointsDeterminePlanes}
Let $P = V+t\in \mathcal{A}(n, k)$, and let $p_0, ...,  p_{k}$ be a set of points in $P$. Suppose $p_0, ...,  p_{k}$ is such that $(p_1 - p_{0}), ...,  (p_{k}-p_{0})$ spans $V\in \mathcal{G}(n, k)$. Then for every oracle $A\subseteq\mathbb{N}$ and $r\in \mathbb{N}$ sufficiently large (depending on $p_0, ...,  p_k$),
\begin{equation*}
    K^B_{r}(P, p_0, ...,  p_k)\leq K^B_r(p_0, ..., p_k) + O_{n, k, \sigma, \vert p_0\vert}(\log r),
\end{equation*}
where $\sigma$ is the smallest nonzero singular value of the matrix with $(p_1-p_{0}), ...,  (p_{k}-p_{0})$ as its columns. 
\end{lem}
The value of $\sigma$ essentially quantifies how degenerate the points are. For instance, one expects a (large but fixed) loss in precision for determining the plane spanned by three nearly collinear points, but this error becomes negligible at higher and higher resolutions.

The next lemma essentially states that the orthogonal complement of a subspace is computable from the subspace. 
\begin{lem}[\cite{fiedler2025extensionsUnions}]\label{lem:planeDeterminesComplement}
    Let $V\in\mathcal{G}(n, k)$, a precision $r$, and $A\subseteq\mathbb{N}$ be given. 
\begin{equation*}
    K_r^B(V^\perp)\leq K_r^B(V) + O_{n, k}(\log r).
\end{equation*}
\end{lem}

Finally, points in a $k$-plane cannot have complexity much higher than $kr$ at precision $r$, conditioned on that $k$-plane at precision $r$. 
\begin{lem}[\cite{fiedler2025extensionsUnions}]\label{lem:pointsOnAPlane}
    Suppose $P=V+t\in \mathcal{A}(n, k)$, and let $x\in P$. Then
\begin{equation*}
       K_r^B(x\mid P)\leq k r + O(\log r). 
\end{equation*}
\end{lem}

\section{Additional lemmas for the Grassmannian}

This section covers two main lemmas. First, we will establish a version of Lemma \ref{lem:CaseLutz} for the Grassmannian. Second, we will prove another geometric lemma which bounds the complexity of a $k_1$-plane conditioned on a larger $k_2$-plane that contains it. 

To see why the first lemma is more challenging in the Grassmannian than $\mathbb{R}^n$, consider that results on Kolmogorov complexity for real numbers extend naturally to real vectors. 
The main reason is that $\R^n$ comes with an obvious, canonical coordinate system: any point in $\R^n$ is simply an $n$-tuple of real numbers. 
We adopt the same philosophy and seek an analogous coordinate description for $k$-planes, even though the structure of the Grassmannian is less intuitive.

Recall that we associate a $k$-plane $V\subseteq \R^n$ to its orthogonal projection matrix $P_V$. While $P_V$ is an $n\times n$ matrix, the space $\G(n,k)$ has dimension $k(n-k)$, so we would like a parametrization with only $k(n-k)$ real degrees of freedom. A slightly better representation is given by a basis matrix $A\in \R^{n\times k}$ whose columns form a basis of $V$. 
This still uses $nk$ entries, but it can be further improved by an important fact: since $\mathrm{rank}(A)=k$, there exists an index set $I\subseteq [n]$ with $|I|=k$ such that the $k\times k$ row submatrix $A_I$ is invertible. In general, we will use the subscript $I$ to denote this row submatrix for a given matrix. 
Define
\[
A' := A\,A_I^{-1}.
\]
Then $A'$ has the same column space as $A$ (hence spans the same $k$-plane), and moreover
\[
A'_I = I_k.
\]
Thus, once the choice of $I$ is fixed (which can be specified with $O(1)$ overhead when $n,k$ are fixed), the remaining degrees of freedom of $A'$ are exactly the entries in the complementary rows, giving $k(n-k)$ free parameters.

However, an arbitrary invertible submatrix is not sufficient in general: we must choose $I$ in a \emph{stable} way so that the new basis $A'$ can be well approximated from an approximation of the underlying $k$-plane. 
The next straightforward observation (see, for instance, the proof of Corollary 23 in \cite{fiedler2025extensionsUnions}) captures the key fact behind this control: for every $V\in \G(n,k)$ there exists a choice of coordinate rows whose associated basis matrix has singular values bounded away from $0$ by a constant depending only on $n$ and $k$. Then based on this, we seek an invertible submatrix $A_I$ for which both $\|A_I\|$ and $\|A_I^{-1}\|$ admit good bounds. 
 \begin{obs}\label{obs:singularValueLowerBound}
        Let $n,k$ be given. Let $\sigma(v_1,\ldots,v_m)$ denote the smallest singular value of the
matrix with $v_1,\ldots,v_m$ as columns and let $e_i$ denote the $i$th standard basis vector. Then there is some constant $\widehat{C}_{n,k}$ such that,
\begin{equation*}
\inf_{V\in \G(n,k)}\ \max_{1\le i_1<\cdots<i_{k}\le n}\ 
\sigma\bigl(Ve_{i_1},..,Ve_{i_k}\bigr)
\ \ge\ \widehat{C}_{n,k}\ >\ 0.
\end{equation*}
    \end{obs}

Note that, since the space of \(n \times n\) matrices is finite dimensional, all matrix norms are equivalent up to constants depending only on \(n\). In particular, after identifying each \(V \in \G(n,k)\) with its orthogonal projection matrix, these constants are harmless in effective dimension arguments. We choose to work with the spectral norm, i.e. the operator norm induced by \(\ell_2\). There are two main reasons for choosing this norm: first, controlling the spectral norm yields useful bounds on singular values, and conversely, bounds on singular values translate into norm bounds; in addition, the metric we use on the Grassmannian is just the spectral norm of the difference of the orthogonal projection matrices:
\[
\rho(V_1,V_2)
=\sup_{x\in\mathbb{R}^n, |x|=1}
\bigl|p_{V_1}x-p_{V_2}x\bigr|
=\|p_{V_1}-p_{V_2}\|_2.
\]

We will also need a few facts in linear algebra to carry out the required computations. For convenience, we collect all of the tools we use in the following lemma.

\begin{lem}\label{lem:linearAlgebra}
    Let $A,A'\in \R^{n \times m}$. Let $\sigma(A)$ denote its smallest singular value, and $\sigma_{\max}(A)$ denote its largest singular value. We then have a few results:
    \begin{enumerate}[label=(\roman*)]
        \item If $n = m$, then $|\det(A)| = \prod_{i = 1}^m\sigma_i(A)$.
        \item $\|A\|_2 = \sigma_{\max}(A)$.
        \item If $n = m$ and $A$ is invertible, then $\|A^{-1}\|_2 = \frac{1}{\sigma(A)}$.
        \item $\min_{\|x\|_2 = 1}\|Ax\|_2 = \sigma(A)$
        \item If $A_I$ is a submatrix of $A$, then $\|A_I\|_2 \leq \|A\|_2$.
        \item If $A = BC$, then $\|A\|_2 \leq \|B\|_2\|C\|_2$
       \item 
       $\|A-A'\|_2 \le \sqrt{\sum_{j \in [m]} \|A_j-A'_j\|_2^2}$,
        where \(A_j\) and \(A'_j\) denote the \(j\)-th columns of \(A\) and \(A'\), respectively.
        \item If $A,A'$ are invertible and $\|A^{-1}\|_2\|A-A'\|_2<1$, then \begin{equation}\label{inverseDistanceBound}
            \|A^{-1} - (A')^{-1}\|_2 \leq \frac{\|A^{-1}\|^2_2\|A-A'\|}{1 - \|A^{-1}\|_2\|A-A'\|_2}.
        \end{equation}
    \end{enumerate}
\end{lem}
\begin{proof}
    Results (i)--(vii) are very standard facts; see, for example, \cite{HornJohnson1985MatrixAnalysis}. Since (viii) is  less immediate, we include a short proof for completeness. Let $E = A' - A$. Then $A' = A + E = A(I + A^{-1}E)$. By assumption, we have $\|A^{-1}E\|_2 \leq \|A^{-1}\|_2\|E\|_2 < 1$. Hence we can apply the Neumann series to $I + A^{-1}E = I - (-A^{-1}E)$ and obtain
    \begin{equation*}
        (I + A^{-1}E)^{-1} = \sum_{k = 0}^\infty (-A^{-1}E)^k.
    \end{equation*}
    Therefore,
    \begin{equation*}
        (A')^{-1} = (I + A^{-1}E)^{-1}A^{-1} = \left(\sum_{k = 0}^\infty (-A^{-1}E)^k\right)A^{-1}.
    \end{equation*}
    Taking norms and again using the bound $\|A^{-1}E\|_2 \leq \|A^{-1}\|_2\|E\|_2$, we get
    \begin{equation}\label{eq:perturbationBound}
        \|(A')^{-1}\|_2 \leq \sum_{k = 0}^\infty \|A^{-1}E\|_2^k \|A^{-1}\|_2
        = \frac{\|A^{-1}\|_2}{1 - \|A^{-1}E\|_2}
        \leq \frac{\|A^{-1}\|_2}{1 - \|A^{-1}\|_2 \|A - A'\|_2}.
    \end{equation}
    Finally, using the identity $A^{-1} - (A')^{-1} = A^{-1}(A' - A)(A')^{-1} = A^{-1}E(A')^{-1}$, we have
    \begin{equation*}
        \|A^{-1} - (A')^{-1}\|_2 \leq \|A^{-1}\|_2\|E\|_2\|(A')^{-1}\|_2.
    \end{equation*}
    Substituting the bound for $\|(A')^{-1}\|_2$ from \eqref{eq:perturbationBound} and recalling that $E = A' - A$ completes the proof.
\end{proof}

In the proofs below, we write $I \in \binom{[n]}{k}$ to indicate that $I \subseteq [n]$ and $|I|=k$. We now assemble the ingredients developed so far into the following lemma, which will be used to build up to our Grassmannian version of Lemma \ref{lem:CaseLutz}.

\begin{lem}\label{lem:goodBasis}
    Let $V \in \G(n,k)$ and $B\subseteq\mathbb{N}$. Then there exists a basis matrix $A \in \R^{n \times k}$ of $V$ (i.e. $A$'s columns form a basis of $V$) and a constant $\widehat{C}_{n,k} > 0$ such that $\|A\|_2 \leq 1$, $\sigma(A) \geq \widehat{C}_{n,k}$, and \begin{equation*}
        K^B_r(p_1,..,p_k\mid V) \leq O_{n,k}(\log r).
    \end{equation*} where $p_i$'s are columns of $A$.
\end{lem}

This preliminary lemma shows that, given a $k$-plane, one can compute a basis matrix whose norm and singular values are uniformly well controlled.

    \begin{proof}
        It is straightforward to construct a Turing Machine $M(\tilde{V}, I)$, where $\tilde{V} \in \G(n,k) \cap \Q^{n \times n} \cap B_{2^{-r}}(V)$ and $I \in \binom{[n]}{k}$  as follows: \begin{itemize}
            \item For all $i \in I$, compute $q_i = \tilde{V}e_i \in \Q^n$.
            \item Return the $q_i$'s.
        \end{itemize}
        By Observation \ref{obs:singularValueLowerBound}, we can find some $I \in \binom{[n]}{k}$ and constant $\widehat{C}_{n,k}$ such that $\sigma(\tilde{A}) \geq \widehat{C}_{n,k}$, where $\tilde{A}$ is the column matrix formed by $q_i$'s. Note that the length of $I$ depends only on $n$ and $k$. Also note that $\tilde{V}$ is an orthogonal projection matrix, so $\|\tilde{V}\|_2 = 1$, which implies\begin{equation*}
            \|\tilde{A}\|_2 = \|\tilde{V}E_I\|_2 \leq \|\tilde{V}\|_2\|E_I\|_2 =1.
        \end{equation*} where $E_I$ is a submatrix of $I_n$ so it also has spectral norm $1$.

        Then, fix $i \in I$. Let $p_i = Ve_i$ and $q_i = \tilde{V}e_i$. By definition of the metric $\rho$, $|e_i| = 1$ implies that  \begin{equation*}
            |p_i - q_i| \leq \rho(V,\tilde{V})  \leq 2^{-r}.
        \end{equation*}
    \end{proof}

To approximate the basis $A' = A\,A_I^{-1}$ defined earlier, we must first choose a ``good'' submatrix $A_I$ whose inverse is well-controlled. 
With the basis matrix provided by the previous lemma, the next lemma guarantees the existence of such a choice of $I$. 
\begin{lem}\label{lem:goodSubmatrix}
 Let $A \in \R^{n \times m}$ have full column rank with $n \geq  m$. Suppose $\|A\|_2 \leq K_1$ and $\sigma(A) \geq K_2 > 0$, where $\sigma$ denotes the smallest singular value. Then there is an invertible $m\times m$ submatrix $A_I$ such that $\|A_I^{-1}\|_2 \leq \frac{\sqrt{\binom{n}{m}}K_1^{m-1}}{K_2^m}$
\end{lem}
\begin{proof}
    Applying the Cauchy-Binet Formula, we have \begin{equation*}
        \det(A^TA) = \sum_{S \in \binom{[n]}{m}}\det(A_S)^2
    \end{equation*}
    So we can find some index set $I\in \binom{[n]}{m}$ such that \begin{equation*}
        \det(A_I)^2 \geq \frac{1}{\binom{n}{m}}\det(A^TA)
    \end{equation*}
    And we also have \begin{equation*}
        \det(A^TA)^{1/2} = \prod_{i = 1}^m \sigma_i(A) \geq \sigma(A)^m \geq K_2^m > 0
    \end{equation*}
    So we get \begin{equation} \label{AJ-detBound}
        |\det(A_I)|\geq \frac{1}{\sqrt{\binom{n}{m}}}K_2^m > 0
    \end{equation}
    On the other hand , since $A_I$ is non-singular, we then have \begin{equation}\label{AJ-svBound}
        |\det(A_I)| = \prod_{i = 1}^m \sigma_i(A_I) \leq \sigma(A_I)\sigma_{\max}(A_I)^{m-1} = \sigma(A_I)\|A_I\|_2^{m-1}
    \end{equation}
    Therefore, we can get \begin{align*}
        \|A_I^{-1}\|_2 & = \frac{1}{\sigma(A_I)} && \text{[Lemma \ref{lem:linearAlgebra} (iii)]}\\
        & \leq \frac{\|A_I\|_2^{m-1}}{|\det(A_I)|} &&\text{[inequality \eqref{AJ-svBound}]} \\
        &\leq  \frac{\|A\|_2^{m-1}}{|\det(A_I)|} && \text{[Lemma \ref{lem:linearAlgebra} (v)]}\\
        &\leq \frac{\sqrt{\binom{n}{m}}K_1^{m-1}}{K_2^m} && \text{[inequality \eqref{AJ-detBound} ]}
    \end{align*}
\end{proof}

 We have already shown that there exists a ``good'' submatrix $A_I$. It just remains to carry out some estimates and verify that the new basis $A' := A\cdot A_I^{-1}$ can be well approximated from an approximation of the underlying $k$-plane.

\begin{lem}\label{lem:coordinates}
    Let $V\in G(n,k)$ and $B\subseteq\mathbb{N}$. Then there exists a basis matrix $A'\in \R^{n\times k}$ whose columns $A'_1,\ldots,A'_k$ form a basis of $V$ and an index set $I\in\binom{[n]}{k}$ such that the $k\times k$ row submatrix satisfies $A'_I = I_k$. Moreover, 
    \begin{equation*}
    K^B_r(A'_1,\ldots,A'_k )\ =\ K^B_r(V) \pm O_{n,k}(\log r).
    \end{equation*}
\end{lem}

We make the following remarks on this lemma. First, it shows that this basis is an effective characterization of a $k$-plane. Moreover, since $A'_I = I_k$, the matrix $A'$ can be identified with a vector in $\R^{k(n-k)}$, which matches the dimension of the Grassmannian $\G(n,k)$. In this sense, $A'$ serves as a ``coordinate" representation of a $k$-plane, allowing us to transfer results of Kolmogorov complexity for real vectors to $k$-planes.

\begin{proof}[Proof of Lemma \ref{lem:coordinates}]
We first show the upper bound.
Given $V \in \G(n,k)$, let $A$ denote the basis in Lemma \ref{lem:goodBasis}. Then we can apply Lemma \ref{lem:goodSubmatrix} to obtain an invertible submatrix $A_I$. Let $A' = A\cdot A_I^{-1}$
    We construct a Turing Machine $M(\tilde{V},I,J)$, where $\tilde{V} \in \Q^{n \times n}\cap \G(n,k) \cap B_{2^{-r}}(V)$, as follows: \begin{itemize}
        \item Use the Turing Machine in Lemma \ref{lem:goodBasis} with inputs $(\tilde{V},J)$ to compute $\tilde{A}$.
        \item Compute $\tilde{A}_I^{-1}$.
        \item Output $\tilde{A}' = \tilde{A} \cdot \tilde{A}_I^{-1}$
    \end{itemize}
    We first need to check that $\tilde{A}_I$ is invertible.  By Lemma \ref{lem:goodBasis}, we know that for each column $j$, $\|A_j -\tilde{A}_j\|_2 \leq 2^{-r}$. Then, we apply Lemma \ref{lem:linearAlgebra} (vii) and get that 
    \begin{equation}\label{A-distance}
     \|A - \tilde{A}\|_2 \leq \sqrt{\sum_{j \in [k]}\|A_j - \tilde{A}_j\|_2^2} \leq \sqrt{k}2^{-r}.
     \end{equation}
    By Lemma \ref{lem:linearAlgebra} (v), we have
     \begin{equation}\label{eq:AJ-distance}
         \|A_I-\tilde{A}_I\|_2\leq \|A - \tilde{A}\|_2 \leq \sqrt{k}2^{-r}.
     \end{equation}
    By Lemma \ref{lem:goodBasis}, $\|A\|_2 \leq 1$ and $\sigma(A) \geq \widehat{C}_{n,k}$. With these two bounds, we apply Lemma \ref{lem:goodSubmatrix} and get that  \begin{equation}\label{eq:AJinverse-normBound}
         \|A_I^{-1}\|_2\leq \frac{\sqrt{\binom{n}{k}}}{(\widehat{C}_{n,k})^k}
     \end{equation}
     which, by Lemma \ref{lem:linearAlgebra} (iii), implies \begin{equation}\label{eq:AI-singularValue}
         \sigma(A_I) = \frac{1}{\|A_I^{-1}\|_2} \geq \frac{(\widehat{C}_{n,k})^k}{\sqrt{\binom{n}{k}}}.
     \end{equation}
      We use the fact that for a matrix $M$, the map $M\mapsto \sigma(M)$ is $1$-Lipschitz with respect to the spectral norm, which gives \begin{equation*}
         |\sigma(A_I) -\sigma(\tilde{A_I})| \leq \|A_I - \tilde{A_I}\|_2  \leq \sqrt{k}2^{-r}. 
     \end{equation*}
      So when $r$ is sufficiently large, by \eqref{eq:AI-singularValue} and previous inequality,  we have \begin{equation*}
        \sigma( \tilde{A_I}) \geq \sigma (A_I) - |\sigma(A_I) -\sigma(\tilde{A_I})|  \geq  \frac{(\widehat{C}_{n,k})^k}{\sqrt{\binom{n}{k}}} - \sqrt{k}2^{-r} > 0
     \end{equation*}
     So $\tilde{A}_I$ is indeed invertible. This allows us to define $\tilde{A}' = \tilde{A}\cdot \tilde{A}_I^{-1}$. 
     
     Now, we show that we can approximate each $A'_j$, i.e. that $\|A_j' - \tilde{A}_j'\|_2$ is well-controlled. First, for $r$ sufficiently large, by \eqref{eq:AJ-distance} and \eqref{eq:AJinverse-normBound}, we have \begin{equation*}
         \|A_I^{-1}\|_2 \|A_I - \tilde{A}_I\| _2 \leq \frac{\sqrt{\binom{n}{k}}}{(\widehat{C}_{n,k})^k} \sqrt{k}2^{-r} < 1/2 < 1
     \end{equation*}
     then by \eqref{inverseDistanceBound}, we have \begin{align}\label{AJinverse-distance}
         \|A_I^{-1} - \tilde{A}_I^{-1}\|_2 \leq \frac{\|A_I^{-1}\|_2^2 \|\tilde{A}_I - A_I\|_2}{1 - \|A_I^{-1}\|_2\|\tilde{A}_I - A_I\|_2}\leq  \frac{2\sqrt{k}\binom{n}{k}}{(\widehat{C}_{n,k})^{2k}}2^{-r}.
     \end{align}
     Then for each column $j$, by the triangle inequality and Lemma \ref{lem:linearAlgebra} (vi), we have \begin{align*}
     \|A'_j - \tilde{A}_j' \|_2 &= \|A(A_I^{-1})_j - \tilde{A} (\tilde{A}_I^{-1})_j\|_2\\
     & \leq \|(A-\tilde{A})(A_I^{-1})_j\|_2 + \|\tilde{A}((A_I^{-1})_j - (\tilde{A}_I^{-1})_j)\|_2\\
     &\leq\|A-\tilde{A}\|_2 \|(A_I^{-1})_j \|_2 + \|\tilde{A}\|_2 \|(A_I^{-1})_j - (\tilde{A}_I^{-1})_j\|_2.
    \end{align*}
    Using the triangle inequality again, together with Lemma \ref{lem:linearAlgebra} (v) and inequality \eqref{A-distance}, we then have \begin{equation*}
         \|A'_j - \tilde{A}_j' \|_2 \leq 
         \sqrt{k}2^{-r}\|(A_I^{-1})\|_2 + (\|A\|_2  + \|A - \tilde{A}\|_2)\|A_I^{-1} - \tilde{A}_I^{-1}\|_2.
    \end{equation*}
    Note that inequality \eqref{A-distance}, the norm bound on $A$, and the fact that $2^{-r} \leq 1$ imply \begin{equation*}
        \|A\|_2  + \|A - \tilde{A}\|_2 \leq 1 + \sqrt{k}2^{-r} \leq 1 + \sqrt{k}.
    \end{equation*}
    Then by inequality \eqref{A-distance} and \eqref{eq:AJinverse-normBound}, we have   
    \begin{align*}
         \|A'_j - \tilde{A}_j' \|_2 
         &\leq  \sqrt{k}2^{-r}\|(A_I^{-1})\|_2 + (1 + \sqrt{k})\|A_I^{-1} - \tilde{A}_I^{-1}\|_2\\ 
         & \leq \left(\frac{\sqrt{k\binom{n}{k}}}{(\widehat{C}_{n,k})^k} + (1  + \sqrt{k})\frac{2\sqrt{k}\binom{n}{k}}{(\widehat{C}_{n,k})^{2k}}\right)2^{-r}.
    \end{align*}
To recap, we have shown that a precision $r$ approximation of $V$ and little additional information is sufficient to compute a nearly precision $r$ approximation of each $A^\prime_j$. Note that the above constant depends only on $n$ and $k$. Hence, applying Lemma \ref{lem:CaseLutz} to each $A_j'$ gives us the desired upper bound. 

For the lower bound on $K^B_r(A'_1,\ldots,A'_k )$, by the fact that $A'_I = I_k$, we have\begin{equation*}
    \|A'x\|_2 \geq \|(A'x)_I\|_2 = \|A'_Ix\|_2 = \|x\|_2.
\end{equation*}
So \begin{equation*}
    \min_{\|x\|_2 = 1}\|A'x\| \geq 1.
\end{equation*}
Then by Lemma \ref{lem:linearAlgebra} (iv), \begin{equation*}
    \sigma(A') \geq 1.
\end{equation*}
Hence, the collection of vectors $A'_j\in V$ is suitably non-singular, so we can apply Lemma \ref{lem:pointsDeterminePlanes} and get the desired lower bound (again depending only on $n$ and $k$).
\end{proof}

We now give a direct application of our characterization by extending Lemma \ref{lem:CaseLutz} from real vectors to $k$-planes. The original statement of Case and Lutz \cite{case2015dimension} gives a sharper error term, but its proof requires a number of delicate estimates. Since our eventual argument proceeds via the point-to-set principle, an $O(\log r + \log s)$ error term is sufficient for our purposes, which allows for a considerably simplified proof.

The idea is straightforward: we just apply Lemma~\ref{lem:CaseLutz} to each free entry of the matrix $A'$ computed from the given $k$-plane $V$. 
\begin{lem}\label{lem:grassmannianCaseLutz}
    Let $V\in \mathcal{G}(n, k)$ and $B\subseteq\mathbb{N}$. Then,
    \begin{equation*}
        K^B_{r+s}(V) \leq K_r(V) + k(n-k)s + O_{n, k}(\log s + \log r)
    \end{equation*}
\end{lem}

\begin{proof}
    Fix \(V \in \G(n,k)\). By Lemma \ref{lem:coordinates} (relativized to $B$), there exist \(A' \in \R^{n \times k}\) with \(A'_I = I_k\) for some \(I \in \binom{[n]}{k}\), such that 
    \begin{equation*}
        K^B_{r+s}(V) \leq K^B_{r+s}(A'_1, \dots, A'_k) + O(\log r + \log s),
    \end{equation*}
    and
    \begin{equation*}
        K^B_r(A'_1, \dots, A'_k) \leq K^B_r(V) + O(\log r),
    \end{equation*}
    where \(A'_1, \dots, A'_k\) are the columns of \(A'\).
    
    Now apply Lemma \ref{lem:CaseLutz} entrywise to the matrix \(A'\). Since \(A'_I = I_k\), the \(k^2\) entries in the rows indexed by \(I\) are fixed rationals, so only the remaining \(k(n-k)\) entries need to be refined from precision \(r\) to precision \(r+s\). Therefore Lemma \ref{lem:CaseLutz} gives
    \begin{equation*}
        K^B_{r+s}(A'_1, \dots, A'_k)
        \leq
        K^B_r(A'_1, \dots, A'_k) + k(n-k)s + O(\log r + \log s).
    \end{equation*}
    Combining the above inequalities completes the proof.
\end{proof}

\noindent \textit{Remark.} This lemma's proof can be modified slightly to obtain an analogous version for affine $k$-planes. Specifically, if $P \in \mathcal{A}(n,k)$, then we can write $P = V + t$ for some $V \in \mathcal{G}(n,k)$ and $t \in \R^n$. One then applies this lemma to $V$ and Lemma \ref{lem:CaseLutz} to $t$.

To end this section, we state and prove a geometric lemma. The intuition behind it is as follows: for $W \in \G(n,k_2)$, we know $W \cong \R^{k_2}$. Then, (working with access to $W$) any $k_1$-plane $V\subseteq W$ can be viewed as living in the Grassmannian $\G(k_2,k_1)$, whose dimension is $k_1 (k_2-k_1)$.
\begin{lem}\label{lem:planesInsidePlanes}
    Let $V\in \mathcal{G}(n, k_1)$ and $W\in \mathcal{G}(n, k_2)$. Assume $k_1<k_2$ and $V\subset W$. Then
    \begin{equation*}
        K^B_r(V\mid W)\leq k_1(k_2 - k_1) r + O(\log r).
    \end{equation*}
\end{lem}
\begin{proof}
First, note that if $V^\prime$ is a subset of any of the standard coordinate $k_2$-planes in $\mathbb{R}^n$, then 
\begin{equation}\label{eq:preliminaryPlanesInsidePlanes}
K^B_r(V^\prime)\leq k_1(k_2 - k_1) r + O(\log r).
\end{equation}
This bound holds because the assumption guarantees that the projection matrix for $V^\prime$ has the form of a projection matrix in $\mathcal{G}(k_2, k_1)$ with $(n-k_2)$ rows and columns consisting entirely of zeros inserted. Expanding any rational projection matrix in this way requires only a constant (depending on $n$ and $k_2$) amount of information, so the upper bound that Lemma \ref{lem:grassmannianCaseLutz} implies for elements of $\mathcal{G}(k_2, k_1)$ also holds for $V^\prime$. 

Now, let $V$ and $W$ be given. Pick a standard coordinate $k_2$-plane that maximizes the smallest nonzero singular value of the projection of its basis elements onto $W$. There exists some $V^\prime$ in this $k_2$ plane such that $p_W V^\prime = V$. Define $v_j=p_W (p_{V^\prime}e_{j})$, where $e_j$ is the $j$th standard coordinate vector in $\mathbb{R}^n$. By our choice of $k_2$-plane and Observation \ref{obs:singularValueLowerBound}, we are guaranteed that some subset $I$ consisting of $k_1$ of these vectors is a basis for $V$ and is such that all of its nonzero singular values are bounded from below by some constant depending only on $n, k_1$, and $k_2$.

Given $2^{-r}$ approximations of $W$ and $V^\prime$ (denoted $\bar{W}$ and $\bar{V^\prime}$) and $I\subseteq[n]$ of cardinality $k_2$, there is a Turing machine that takes the standard basis vectors in $\mathbb{R}^n$ corresponding to $j\in I$, then computes $\bar{v}_j = p_{\bar{W}} (p_{\bar{V^\prime}}e_{j})$ for each. By the definition of the metric on the Grassmannian, $\vert v_j - \bar{v}_j\vert\leq 2^{-(r-1)}$. Hence (with the correct choice of $I$), we may use Lemma \ref{lem:pointsDeterminePlanes} to establish
\begin{equation*}
    K^B_r(V\mid W)\leq K^B_r(V^\prime) + O(\log r).
\end{equation*}
Applying \eqref{eq:preliminaryPlanesInsidePlanes} completes the proof.
\end{proof}

\section{Enumeration lemma}

The following lemma generalizes Lemma 3.1 of \cite{LutStu18Projections}. The second condition says that if \(w\) has the same projection as \(z\), then either \(w\) is close to \(z\) or \(w\) has high complexity. Consequently, by enumerating all low-complexity points with the same projection as \(z\), one can recover \(z\). The first condition guarantees that the corresponding enumeration Turing machine always halts and returns a valid candidate. In the proof of our main theorem, the goal will be to check the conditions on the lemma so we can apply it and deduce a strong bound on the complexity of a projected point. 

\begin{lem}\label{lem:enumeration}
Suppose that $z \in \mathbb{R}^n$, $V \in \G(n,k)$, $r \in \mathbb{N}$, $\delta \in \mathbb{R}_{+}$, and
$\ve,\eta \in \mathbb{Q}_{+}$ satisfy $r \geq \log\!\bigl(2\|z\|+5\bigr)+10$
and the following conditions.
\begin{enumerate}
  \item $K_r(z) \le (\eta+\ve)r $.
  \item For every $w \in B_2(z)$ such that $p_V w = p_V z$,
  \[
  K_r(w) \ge (\eta-\ve)r + (r-t)\delta,
  \]
  whenever $t = -\log \|z-w\| \in (0,r]$.
\end{enumerate}
Then for every oracle set $A \subseteq \mathbb{N}$,
\[
K_{r}^{A,V}(p_V z) \ge K_{r}^{A,V}(z)
  - \frac{2n\ve}{\delta}\,r - K(\ve) - K(\eta) - O(\log r),
\]
where the constant implied by the big-$O$ bound depends only on $z$, $p_V$, and $n$.
\end{lem}

We will need the following observation of Lutz and Stull. 
\begin{obs}[\cite{LutStu18Projections}]\label{obs:enumeration}
    Let $z \in \mathbb{R}^n$, $p \in \mathbb{Q}^n$, $V \in \G(n,k)$, and $r \in \mathbb{N}$ such that $\| p_V z - p_V p\| \le 2^{-r}$. Then there is a $w \in \mathbb{R}^n$ such that $\|p-w\| \le 2^{-r}$ and $p_Vz =p_V w$.
\end{obs}
\noindent Our formulation is a slight generalization of theirs, replacing projections onto lines by projections onto $k$-planes, but the proof is identical. 

More generally, the proof of this enumeration lemma is very similar to \cite{LutStu18Projections}. However, we include it for the purpose of completeness. 

\begin{proof}[Proof of Lemma \ref{lem:enumeration}]
    We design an oracle Turing Machine $M^{A,V}(\widetilde{p_Vz},\tilde{z} ,\tilde{r},\ve,\eta)$, where $\eta,\ve \in \Q$, $\widetilde{p_Vz}\in \Q^n \cap B_{2^{-r}}(p_V z)$, $\tilde{z}  \in \Q^n \cap B_{2^{-10}}(z)$ and $\tilde{r} = r-10 \in \N$, as follows: \begin{itemize}
        \item For every program $\sigma \in \{0,1\}^*$ with $\ell(\sigma) \leq (\eta + \ve)(r-10)$, in parallel, $M$ simulates $U(\sigma)$.
        \item If one of the simulations halts with output $p \in \Q^n\cap B_{2^{-1}}(\tilde{z} )$ such that $\|p_Vp - \widetilde{p_Vz}\|\leq 2^{-(r-10)}$, then $M^{A,p_V}$ halts with output $p$. 
    \end{itemize}
    Note that by assumption (i), there is some $\sigma$ such that $U(\sigma) = \bar{z} \in \Q^n$ where $\|\Bar{z} - z \| \leq 2^{-r}$. Then \begin{equation*}
        \|\Bar{z} - \tilde{z}\| \leq \|\Bar{z} - z \| + \|z - \tilde{z}\| \leq 2^{-r} + 2^{-10} \leq 2^{-1},
    \end{equation*} which implies $\bar{z} \in \Q^n \cap B_{2^{-1}}(\tilde{z})$. Since the function $x \mapsto p_V x$ is $1$-Lipschitz, we have \begin{equation*}
    \|p_V \bar{z} - \widetilde{p_V z}\| \leq \|p_V \bar{z} - p_V z\| +\|p_V z - \widetilde{p_V z}\| \leq 2^{-r} + 2^{-r} \leq 2^{-(r-10)}
    \end{equation*} 
    So $M^{A,V}$ is guaranteed to halt on our input. As for the output $p$, we have \begin{equation*}
        \|p_V p - p_V z\| \leq \|p_V p - \widetilde{p_V z}\| + \|\widetilde{p_Vz} - p_V z\| \leq 2^{-(r-10)} + 2^{-r}\leq 2^{-r + 1}
    \end{equation*}
     So we know, by Observation \ref{obs:enumeration}, there is some \begin{equation*}
         w \in B_{2^{-r+1}}(p) \subset B_{2^{-1}}(p)\subset B_{2^0}(\tilde{z})\subset B_{2^{1}}(z)
     \end{equation*} such that $p_V w = p_V z$. Therefore, by Lemma \ref{lem:CaseLutz}, we have \begin{align*}
         K^{A,V}_r(w) &\leq K_r^{A,V}(p_V z) + K_{10}(z) + K(r) + K(\ve) + K(\eta) + O(\log r)\\
        & \leq K_r^{A,V}(p_V z)  + K(\ve) + K(\eta) + O(\log r).
     \end{align*}
     After rearranging, we get \begin{equation}\label{eq:p_Vz-lowerBound}
         K_r^{A,V}(p_V z) \geq K_r^{A,V}(w) - K(\ve) - K(\eta) - O(\log r).
     \end{equation}
     Let $t = -\log \|z - w\|$. If $t \geq r$, then the proof is already complete. If $t < r$, $B_{2^{-r + 1}}(p) \subset B_{2^{-t + 2}}(z)$; together with Lemma \ref{lem:CaseLutz}, we have\begin{equation}\label{eq:w-lowerBound}
         K_r^{A,V}(w)\geq K_{r }^{A,V}(z) - n(r - t) - O(\log r),
     \end{equation}
     and by construction, we have \begin{align*}
         (\eta + \ve)r \geq K(p)\geq K_r(w) - O(\log r).
     \end{align*}
     By assumption (ii), we then have \begin{equation*}
         (\eta + \ve)r \geq (\eta - \ve)r + (r- t)\delta
     \end{equation*}
     which then implies\begin{equation*}
         r -t \leq \frac{2\ve}{\delta}r + O(\log r)
     \end{equation*}
     Combining this with the inequalities \eqref{eq:p_Vz-lowerBound} and \eqref{eq:w-lowerBound} completes the proof.
\end{proof}

\section{Geometric lemma}
The following lemma generalizes Lemma 3.3 of \cite{LutStu18Projections}.  The intuition is that two points with an identical projection onto a subspace will either be very close together, or give a significant amount of information about that subspace. As compared to the original lemma in \cite{LutStu18Projections}, this proof will require several distinct geometric tools, which we have introduced in previous sections. 
\begin{lem}\label{lem:geometric}
    Let $z,w \in \R^n$, $B\subseteq\mathbb{N}$, and $V \in \G(n,k)$ be such that $p_V z = p_V w$ and $z$ agrees with $w$ up to precision $t$. Then, 
\begin{equation*}
    K^B_r(w)\geq K^B_t(z) + K^B_{r-t, r}(V\mid z) - k(n-1-k)(r-t) -O(\log r)
\end{equation*}
\end{lem}
\begin{proof}
    Since $z$ and $w$ agree up to precision $t$,
\begin{align*}
    K^B_r(w)&= K^B_{r, t}(w\mid w) + K^B_t(w)\pm O(\log r)\\
    &= K^B_{r, t}(w\mid z) + K^B_t(z) \pm O(\log r)\\
    &\geq K^B_{r}(w\mid z) + K^B_t(z) -O(\log r)
\end{align*}
    So it suffices to show that
    \begin{equation}\label{eq:mainGeoLemma}
        K^B_r(w\mid z)\geq K^B_{r-t, r}(V\mid z) - k(n-1-k)(r-t) -O(\log r)
    \end{equation}

Using Lemma \ref{lem:pointsDeterminePlanes}, it is easy to see that
\begin{equation*}
    K^B_r(w\mid z)\geq K^B_{r-t, r}(\ell\mid z) - O(\log r)
\end{equation*}
where $\ell$ is the line through $w$ and $z$. Reading off the direction of $\ell$ and employing Lemma \ref{lem:planeDeterminesComplement} gives the same bound for $H$, the hyperplane onto which $z$ and $w$ have the same projection. In particular,
\begin{equation}\label{eq:HfromZ}
    K^B_r(w\mid z)\geq K^B_{r-t, r}(H\mid z) - O(\log r)
\end{equation}
By assumption, $V\subseteq H$, so by Lemma \ref{lem:planesInsidePlanes},
\begin{equation}\label{eq:VinH}
    K^B_{r-t}(V\mid H)\leq k(n-1-k)(r-t)+O(\log r).
\end{equation}

At worst, one could compute $V$ from $z$ by passing through $H$, so
\begin{equation*}
K^B_{r-t, r}(V\mid z) \leq K^B_{r-t, r}(H\mid z) + K^B_{r-t}(V\mid H) + O(\log r).
\end{equation*}
Rearranging and using \eqref{eq:HfromZ} and \eqref{eq:VinH} gives
\begin{equation*}
K^B_r(w\mid z)\geq K^B_{r-t, r}(V\mid z) - k(n-1-k)(r-t) -O(\log r),
\end{equation*}
which is exactly \eqref{eq:mainGeoLemma}
\end{proof}

\section{Proof of main theorem}
With the tools developed in the previous sections, we are now ready to prove the main theorems. We begin with the following theorem, which gives a strong quantitative lower bound on the dimension of a projection under a suitable assumption on the complexity of the \(k\)-plane. From this theorem, we will be able to deduce the desired exceptional set estimate.

\begin{thm}\label{thm:1}
    Let $F \subseteq \R^n$ with optimal oracle $A$. Let $V \in \G(n,k)$. If $F$ has Hausdorff dimension $a$, and $\dim^A(V) \geq b$ for some $b > k(n-1-k)$, then \begin{equation*}
        \dim_H(p_VF) \geq \min\{a, b - k(n-1-k)\}
    \end{equation*}
\end{thm}

The proof proceeds as follows. First, we choose a point \(z \in F\) of high complexity relative to all of the objects in the problem.  We then use Lemma \ref{lem:oracleD} to control the complexity of \(z\), which yields the first condition of Lemma \ref{lem:enumeration}. Next, we apply Lemma \ref{lem:geometric} to obtain the second condition of Lemma \ref{lem:enumeration}.  Finally, applying Lemma \ref{lem:enumeration}, we conclude that \(p_V z\) also has high complexity.

\begin{proof}
    Fix an arbitrary oracle $B \subseteq \N$. By the point-to-set principle, it suffices to show that \begin{equation}\label{eq:thm1LowerBound}
        \sup_{z \in F} \dim^B(p_V z) \geq \min\{a,b - k(n-1-k) \}.
    \end{equation}
    Fix $\ve > 0$. By the definition of optimal oracles, we know that there exists $ z \in F$ such that $\dim^{A,B,V} (z)\geq a  - \frac{\ve}{2}$ and for all sufficiently large $r \in \N$, \begin{equation}\label{eq:zOptimalOracle}
        K_r^{A,B,V} (z) \geq K_r^A(z) -\frac{\ve}{2} r.
    \end{equation}
    We will check the conditions of Lemma \ref{lem:enumeration}, relative to an oracle. Let $\eta \in \Q\cap (0, \min \{\dim^A(z), b - k(n-1-k)\}- \sqrt{\ve})$ and let $D$ be the oracle of Lemma \ref{lem:oracleD} with respect to this $\eta$ and $A$. By property (i) of Lemma \ref{lem:oracleD}, \begin{equation*}
        K_r^{A,D}(z) \leq \eta r + O(\log r)
    \end{equation*}
    Then for $r$ sufficiently large, we have \begin{equation*}
        K_r^{A,D}(z) \leq (\eta  + \ve)r
    \end{equation*}
    which is exactly the condition (i) of Lemma \ref{lem:enumeration}. To obtain condition (ii) of Lemma \ref{lem:enumeration}, we apply Lemma \ref{lem:geometric}. Let $w \in B_2(z)$ such that $p_Vz = p_Vw$ and $\|z-w\| \geq 2^{-r}$. Let $t = - \log \|z-w\|$. By Lemma \ref{lem:geometric} relative to $(A,D)$,  we have \begin{equation}\label{eq:wLowerBound}
        K_r^{A,D}(w) \geq K_t^{A,D}(z) + K_{r-t,r}^{A,D}(V \mid z) - k(n-1-k)(r-t) - O(\log r)
    \end{equation}
    We know that, \begin{align*}
        K_{r,r-t}^{A}(z\mid V) &\geq K_r^{A,V}(z) - O(\log r) && \text{[Lemma \ref{lem:conditionalAndOracle}]}\\
        &\geq K^{A,B,V}_r(z) - O(\log r) && \text{[Lemma \ref{lem:conditionalAndOracle}]}\\
        &\geq K_r^{A}(z) -\frac{ \ve}{2} r - O(\log r)&& \text{[inequality \eqref{eq:zOptimalOracle}]}
    \end{align*}
    Therefore, \begin{align*}
        K_{r-t,r}^{A,D}(V \mid z) 
        &=  K_{r-t,r}^A(V \mid z) \pm O(\log r) && \text{[Lemma \ref{lem:oracleD}]} \\
        &=K_{r,r-t}^{A}(z \mid V) + K_{r-t}^{A}(V) - K_r^{A}(z) \pm O(\log r)&& \text{[Proposition \ref{prop:symmetry}]}\\
        & \geq K_{r-t}^A(V) - \frac{ \ve}{2} r - O(\log r).
    \end{align*} 
    Note that for $r - t \leq \log r$, $K_{r-t}^{A}(V) \geq b(r-t) - o(r)$ trivially; if $r - t > \log r$, by the definition of effective dimension and by requiring $r$ sufficiently large, we also have $K_{r-t}^A(V) \geq b(r - t) - o(r)$. So for all $t\leq r$, we have that \begin{equation}\label{eq:VGivenz}
          K_{r-t,r}^{A,D}(V \mid z)  \geq b(r-t)- \frac{ \ve}{2} r - o(r)
    \end{equation}
    Similarly, by requiring $r$ sufficiently large, for any $ t\leq r$, 
    \begin{equation}\label{eq:zLowerbound}
        K_t^{A,D}(z) \geq \eta t - o(r)
    \end{equation}
   Plugging \eqref{eq:VGivenz} and \eqref{eq:zLowerbound} back into \eqref{eq:wLowerBound},
   \begin{align*}
         K_r^{A,D}(w) &\geq \eta t + b(r-t)- k(n-1-k)(r-t)  - \frac{\ve}{2} r - o(r)\\
         & \geq \eta t+ b(r-t)- k(n-1-k)(r-t)  -\ve r && [\text{$r$ sufficiently large}] \\
         &= (\eta - \ve)r + (r-t )\delta &&[\text{rearranging}]
   \end{align*}
   where $\delta := b - k(n-1-k) - \eta\geq \sqrt{\epsilon} > 0 $. Note that this is exactly condition (ii) of Lemma \ref{lem:enumeration}, relative to $(A,D)$. So we now have that
   \begin{align*}
        K_r^B(p_Vz)&\geq K_r^{A,B,D,V}(p_V z) - O(\log r) && \text{[Lemma \ref{lem:conditionalAndOracle}]}\\
        &\geq K_r^{A,B,D,V}(z) - \ve r - \frac{2n\ve}{\delta}r  - O(\log r)
        &&\text{[Lemma \ref{lem:enumeration}, relative to $(A,D)$]}\\ 
        &\geq K_r^{A,D}(z) - 2\ve r - \frac{2n\ve}{\delta}r  - O(\log r)&&\text{[Lemma \ref{lem:oracleD2}, relative to $A$]}  \\
        &\geq \eta r  - 2\ve r - \frac{2n\ve}{\delta}r - o(r)&& \text{[inequality \eqref{eq:zLowerbound}]}
    \end{align*}
    Then taking the limit inferior, we get \begin{equation*}
        \dim^{B}(p_V z) \geq \eta - 2\ve - \frac{2n\ve}{\delta}
    \end{equation*}

    We now bound the error in this inequality in terms of $\ve$. We were free to pick $\eta$ in a certain interval, and in particular we may assume $\eta \geq \min \{\dim^A(z), b - k(n-1-k)\}-2 \sqrt{\ve}$. Recalling $\delta \geq \sqrt{\ve}$, we have \begin{equation*}
        \dim^B(p_Vz) \geq 
        \min\{\dim^A(z), b - k(n-1-k)\}-2 \sqrt{\ve}
        -2\ve - 2n\sqrt{\ve}
    \end{equation*}
    Letting  $\ve$ approach $0$ gives \eqref{eq:thm1LowerBound}.
\end{proof}

With Theorem \ref{thm:1}, we are  able to prove Theorem \ref{thm:exceptionalSetEstimate}, i.e. a Kaufman type exceptional set estimate for sets with optimal oracles. We restate it here for convenience.
\begin{T2}
      Let $F\subseteq\mathbb{R}^n$ with optimal oracles and $k<n$ be given. Define
    \begin{equation*}
        E_s(F)=\{V\in\mathcal{G}(n, k): \dim_H(p_V F)<s\}.
    \end{equation*}
    If $F$ has Hausdorff dimension $a \geq s$, then
    \begin{equation*}
        \dim_H(E_s(F))\leq k(n-k) + s - k
    \end{equation*}
\end{T2}
\begin{proof}
    Fix an optimal oracle $A$ for $F$. Suppose for the sake of contradiction that $\dim_H(E_s(F)) > k(n-k) + s - k$. Then by the point-to-set principle, we can find $V \in E_s(F)$ such that \begin{equation*}
        \dim ^A(V) \geq k(n-k) + s - k  + \ve
    \end{equation*} with some $\ve > 0$ sufficiently small.
    By Theorem \ref{thm:1}, we have \begin{align*}
        \dim_H(p_VF) &\geq \min\{a, k(n-k) + s- k + \ve - k(n-1-k)\}\\
        & = \min\{a,s + \ve\}
    \end{align*}
    Given $ a\geq s$, we then have \begin{equation*}
        \dim_H(p_VF) \geq s
    \end{equation*}
    which contradicts the fact that $V \in E_s(F)$.
\end{proof}

Now, our generalization of Marstrand's Projection Theorem for sets with optimal oracles follows quickly as a corollary of this estimate. We restate it here for convenience.

\begin{T1}
    Let $F \subseteq \R^n$ with optimal oracles. Then for almost every $V \in \G(n,k)$, \begin{equation*}
        \dim_H(p_VF)= \min\{\dim_H (F), k\}
    \end{equation*}
\end{T1}
\begin{proof}
    The upper bound is trivial; for the lower bound, let $a = \dim_H(F)$. Define
    \begin{equation*}
        E_m = \{V\in \G(n,k): \dim_H(p_VF) < \min\{a,k\} - \frac{1}{m}\}.
    \end{equation*}
    Note that $ \min\{a,k\} - \frac{1}{m} \leq a$, so by Theorem \ref{thm:exceptionalSetEstimate}
    we get that \begin{equation*}
        \dim_H(E_m(F))  \leq k(n-k) + \min\{a,k\} - \frac{1}{m} - k < k(n-k).
    \end{equation*}
    In particular, each of these sets has measure zero. Hence, $E = \bigcup_{m \in \N} E_m$, which is the set of exceptional directions, has measure zero in $\G(n,k)$, which implies for almost every $V \in \G(n,k)$\begin{equation*}
        \dim_H(p_VF) \geq \min\{a,k\}.
    \end{equation*}
 \end{proof}

 \section*{Acknowledgments}
 The authors would like to thank Don Stull for reviewing an earlier draft of this manuscript. 

\bibliographystyle{amsplain}
\bibliography{references}

@article{Kaufman1968ExceptionalSetEstimate,
author = {Kaufman, Robert},
title = {On {H}ausdorff dimension of projections},
journal = {Mathematika},
volume = {15},
number = {2},
pages = {153-155},
doi = {https://doi.org/10.1112/S0025579300002503},
url = {https://londmathsoc.onlinelibrary.wiley.com/doi/abs/10.1112/S0025579300002503},
eprint = {https://londmathsoc.onlinelibrary.wiley.com/doi/pdf/10.1112/S0025579300002503},
year = {1968}
}

@article{Marstrand1954Projection,
author = {Marstrand, J. M.},
title = {Some Fundamental Geometrical Properties of Plane Sets of Fractional Dimensions},
journal = {Proceedings of the London Mathematical Society},
volume = {s3-4},
number = {1},
pages = {257-302},
doi = {https://doi.org/10.1112/plms/s3-4.1.257},
url = {https://londmathsoc.onlinelibrary.wiley.com/doi/abs/10.1112/plms/s3-4.1.257},
eprint = {https://londmathsoc.onlinelibrary.wiley.com/doi/pdf/10.1112/plms/s3-4.1.257},
year = {1954}
}

@book{HornJohnson1985MatrixAnalysis, 
place={Cambridge}, 
title={Matrix Analysis}, 
publisher={Cambridge University Press}, 
author={Horn, Roger A. and Johnson, Charles R.}, 
year={1985}}

@article{case2015dimension,
  title={Mutual dimension},
  author={Case, Adam and Lutz, Jack H.},
  journal={ACM Trans.~Comput.~Theory},
  volume={7},
  number={3},
  pages={1--26},
  year={2015}
}

@article{CholakCsorn2025Bourgain,
 title={Algorithmic Information Bounds for Distances and Orthogonal Projections}, 
      author={Peter Cholak and Marianna Csörnyei and Neil Lutz and Patrick Lutz and Elvira Mayordomo and D. M. Stull},
      year={2025},
      journal={arXiv preprint 2509.05211},
      eprint={2509.05211},
      archivePrefix={arXiv},
      primaryClass={cs.CC},
      url={https://arxiv.org/abs/2509.05211}, 
}

@article{CroFishJack22, 
    title={Hausdorff dimension regularity properties and games},
    volume={248}, 
    url={https://doi.org/10.1007/s11856-022-2299-1},
    journal={Israel Journal of Mathematics},
    author={Crone, Logan and Fishman, Lior and Jackson, Stephen},
    pages={481-500},
    issue={1},
    year={2022},
    month={5}
}

@article{CsornStull2025A,
      title={Improved Bounds for Radial Projections in the Plane}, 
      author={Cs\"ornyei, Marianna and Stull, D. M.},
      journal={arXiv preprint 2508.18228},
      year={2025},
      eprint={2508.18228},
      archivePrefix={arXiv},
      primaryClass={math.CA},
      note={arXiv:2508.18228}, 
}

@book{downey2010,
  title     = "Algorithmic Randomness and Complexity",
  author    = "Downey, Rodney G. and Hirschfeldt, Denis R.",
  year      = 2010,
  publisher = "Springer",
  address   = "New York"
}

@article{Falconer82Projections, title={Hausdorff dimension and the exceptional set of projections}, volume={29}, DOI={10.1112/S0025579300012201}, number={1}, journal={Mathematika}, author={Falconer, K. J.}, year={1982}, pages={109–115}}

@misc{falconer2026seventyyearsfractalprojections,
      title={Seventy Years of Fractal Projections}, 
      author={Kenneth J. Falconer},
      year={2026},
      eprint={2602.22002},
      archivePrefix={arXiv},
      primaryClass={math.MG},
      url={https://arxiv.org/abs/2602.22002}, 
}

@incollection{FalFraJin15,
	author = {Falconer, Kenneth and Fraser, Jonathan and Jin, Xiong},
	booktitle = {Fractal geometry and stochastics V},
	pages = {3--25},
	publisher = {Springer},
	title = {Sixty years of fractal projections},
	year = {2015}}

@article{fiedler2025extensionsUnions,
  title={On the packing dimension of unions and extensions of $k$-planes},
  author={Fiedler, Jacob B.},
  journal={arXiv preprint 2508.18257},
  eprint={2508.18257},
  archivePrefix={arXiv},
  primaryClass={math.CA},
  year={2025},
  url={https://arxiv.org/abs/2508.18257},
     note={arXiv:2508.18257}
}

@article{FieStu24Projections,
      title={Universal sets for projections}, 
      author={Fiedler, Jacob B. and Stull, D. M.},
      journal={arXiv preprint 2411.16001},
      year={2024},
      eprint={2411.16001},
      archivePrefix={arXiv},
      primaryClass={math.CA},
      note={arXiv:2411.16001}, 
}

@article{gan2024exceptionalsets,
    author = {Gan, Shengwen},
    title = {Exceptional Set Estimate Through {B}rascamp–{L}ieb Inequality},
    journal = {International Mathematics Research Notices},
    volume = {2024},
    number = {9},
    pages = {7944-7971},
    year = {2024},
    month = {02},
    issn = {1073-7928},
    doi = {10.1093/imrn/rnae008},
    url = {https://doi.org/10.1093/imrn/rnae008},
    eprint = {https://academic.oup.com/imrn/article-pdf/2024/9/7944/57442834/rnae008.pdf},
}

@book{li2008introduction,
  title={An Introduction to {K}olmogorov Complexity and Its Applications},
  author={Li, Ming and Vit\'{a}nyi, Paul},
  edition={3},
  year={2008},
  city={New York},
  publisher={Springer}
}

@article{LutStu18Projections,
title = {Projection theorems using effective dimension},
journal = {Information and Computation},
volume = {297},
pages = {105137},
year = {2024},
issn = {0890-5401},
doi = {https://doi.org/10.1016/j.ic.2024.105137},
url = {https://www.sciencedirect.com/science/article/pii/S0890540124000026},
author = {Neil Lutz and D. M. Stull}
}

@article{Lutz03b,
	author = {Jack H. Lutz},
	journal = {Inf. Comput.},
	number = {1},
	pages = {49--79},
	title = {The dimensions of individual strings and sequences},
	volume = {187},
	year = {2003}
}

@article{lutz2018algorithmic,
  title={Algorithmic information, plane {K}akeya sets, and conditional dimension},
  author={Lutz, Jack H. and Lutz, Neil},
  journal={ACM Trans.~Comput.~Theory},
  volume={10},
  number={2},
  pages={1--22},
  year={2018}
}

@article{LuLuMay2023PtS,
title = {Extending the reach of the point-to-set principle},
journal = {Information and Computation},
volume = {294},
pages = {105078},
year = {2023},
issn = {0890-5401},
doi = {https://doi.org/10.1016/j.ic.2023.105078},
url = {https://www.sciencedirect.com/science/article/pii/S0890540123000810},
author = {Jack H. Lutz and Neil Lutz and Elvira Mayordomo}
}

@article{lutz2020bounding,
  title={Bounding the dimension of points on a line},
  author={Lutz, Neil and Stull, Donald M.},
  journal={Inform.~and Comput.},
  volume={275},
  pages={104601},
  year={2020}
}

@article{mattila1975exceptionalset, 
    title={Hausdorff dimension, orthogonal projections and intersections with planes}, 
    volume={1}, 
    url={https://afm.journal.fi/article/view/134257}, 
    DOI={10.5186/aasfm.1975.0110}, 
    number={2}, 
    journal={Annales Fennici Mathematici}, 
    author={Mattila, Pertti}, 
    year={1975},
    month={Aug.}, 
    pages={227–244} 
}

@article{Mayordomo02,
	author = {Elvira Mayordomo},
	journal = {Inf. Process. Lett.},
	number = {1},
	pages = {1--3},
	title = {A {K}olmogorov complexity characterization of constructive {H}ausdorff dimension},
	volume = {84},
	year = {2002}}

@article{Orponen2021Projections, title={Combinatorial proofs of two theorems of {L}utz and {S}tull}, volume={171}, DOI={10.1017/S0305004120000328}, number={3}, journal={Mathematical Proceedings of the Cambridge Philosophical Society}, author={Orponen, Tuomas}, year={2021}, pages={503–514}
}

@article{ren2023furstenberg,
  title={{F}urstenberg sets estimate in the plane},
  author={Ren, Kevin and Wang, Hong},
  journal={arXiv preprint 2308.08819},
  eprint={2308.08819},
  archivePrefix={arXiv},
  primaryClass={math.CA},
  year={2023}
}

@inproceedings{Stull22a,
	author = {Stull, D. M.},
	booktitle = {39th International Symposium on Theoretical Aspects of Computer Science (STACS 2022)},
	isbn = {978-3-95977-222-8},
	issn = {1868-8969},
	pages = {57:1--57:17},
	series = {Leibniz International Proceedings in Informatics (LIPIcs)},
	title = {{Optimal Oracles for Point-To-Set Principles}},
	volume = {219},
	year = {2022}}

@article{stull2022pinned,
  title={Pinned distance sets using effective dimension},
  author={Stull, Donald M.},
  journal={arXiv preprint 2207.12501},
  eprint={2207.12501},
  archivePrefix={arXiv},
  primaryClass={cs.CC},
  year={2022}
}
\end{document}